\documentclass[amsppt,11pt]{amsart}
\usepackage{amsmath}
\usepackage{amsfonts}
\usepackage{amssymb}

\newcommand{\ep}{{\varepsilon}}

\newcommand{\intg}{{\mathbb Z}}
\newcommand{\cplx}{{\mathbb C}}

\newcommand{\real}{{\mathbb R}}
\newcommand{\fifi}{{\mathbb F}_q}
\newcommand{\fifibar}{\bar{\mathbb F}_q}

\newcommand{\fifisq}{{\mathbb F}_{q^2}}

\newcommand{\cP}{{\mathcal P}}
\newcommand{\cX}{{\mathcal X}}
\newcommand{\cS}{{\mathcal S}}

\newcommand{\bG}{{\bf G}}

\newcommand{\U}{\mathrm{U}}
\newcommand{\GL}{\mathrm{GL}}
\newcommand{\btau}{\boldsymbol{\tau}}
\newcommand{\bmu}{\boldsymbol{\mu}}
\newcommand{\blam}{\boldsymbol{\lambda}}
\newcommand{\bnu}{\boldsymbol{\nu}}
\newcommand{\spanning}{\mathrm{-span}}
\newcommand{\ch}{\mathrm{ch}}

\newtheorem{thm}{Theorem}[section]
\newtheorem{lemma}{Lemma}[section]
\newtheorem{prop}{Proposition}[section]

\def\adots{\mathinner{\mkern2mu\raise0pt\hbox{.}  
\mkern2mu\raise4pt\hbox{.}\mkern1mu
\raise7pt\vbox{\kern7pt\hbox{.}}\mkern1mu}}

\numberwithin{equation}{section}

\begin{document}

\bibliographystyle{ieeetr}

\title[Semisimple symplectic characters of finite unitary groups]
{Semisimple symplectic characters of finite unitary groups}
\author{C. Ryan Vinroot}
  \address{Mathematics Department\\
           The College of William and Mary\\
           Williamsburg, VA  23187-8795}
   \email{vinroot@math.wm.edu}

\begin{abstract}
Let $G = \U(2m, \fifisq)$ be the finite unitary group, with $q$ the power of an odd prime $p$.  We prove that the number of irreducible complex characters of $G$ with degree not divisible by $p$ and with Frobenius-Schur indicator $-1$ is $q^{m-1}$.  We also obtain a combinatorial formula for the value of any character of $\U(n, \fifisq)$ at any central element, using the characteristic map of the finite unitary group.
\end{abstract}

\maketitle

\section{Introduction}

Let $\U(n,\fifisq)$ denote the finite unitary group defined over the finite field $\fifi$ with $q$ elements, where $q$ is a power of the prime $p$.  A {\em semisimple} character of $\U(n,\fifisq)$ is an irreducible complex character with degree prime to $p$.  If an irreducible complex character $\chi$ of a finite group $G$ is real-valued, then $\chi$ is called {\em symplectic} if its associated complex representation cannot be defined over the real numbers (otherwise, the real-valued character $\chi$ is called {\em orthogonal}; see Section \ref{FS}).  In \cite{GoVi08}, it was conjectured that when $q$ is odd, the group $\U(2m, \fifisq)$ has exactly $q^{m-1}$ semisimple symplectic characters.  The main result of this paper is the proof of this conjecture.  

In Section \ref{Pre}, we give definitions and notation to be used in the rest of the paper.  In Section \ref{Char}, we give a combinatorial characterization of the irreducible characters of the finite unitary group.  In particular, Theorem \ref{Characters} gives a map between the character ring of the finite unitary group and a ring of symmetric functions, coming from the development in \cite{ThVi07}.  This map is the primary combinatorial tool used in the rest of the paper.  In Section \ref{CentralCharacters}, we obtain a formula for the value of any character of the finite unitary group at any central element, in Theorem \ref{CentralValues}, which is analogous to the corresponding formula for the finite general linear group.  In Section \ref{RegSemi}, we give results which tell when a semisimple (or regular) real-valued character of $\U(n, \fifisq)$ is symplectic or orthogonal, and relates this to the value of the character at a specific element of the center, which is exactly where we apply Theorem \ref{CentralValues}.  Section \ref{Uni} is a bit of an aside which explains why the results of Section \ref{RegSemi} do not extend to all characters of the finite unitary group.

Finally, the main result is proven in Sections \ref{SelfDual} and \ref{Main}.  In Section \ref{SelfDual}, a bijection is established between the real-valued semisimple characters of $\U(n, \fifisq)$ and the self-dual polynomials in $\fifi[x]$ of degree $n$ (when $q$ is odd).  In Section \ref{Main}, after several Lemmas, it is shown that the semisimple symplectic characters of $\U(2m, \fifisq)$, $q$ odd, correspond to the self-dual polynomials in $\fifi[x]$ of degree $2m$ with constant term $-1$, and the proof of Theorem \ref{MainThm} is completed by counting these polynomials.

The context of the main result of this paper, and the reason why it was originally conjectured, are as follows.  Let $\tau$ denote the transpose-inverse automorphism of the group $U_n = \U(n, \fifisq)$, so that $\tau(g) = {^T g}^{-1}$ for any $g \in \U(n,\fifisq)$, and let $U_n \langle \tau \rangle$ denote the split extension of the group $U_n$ by the order $2$ automorphism $\tau$.  Then, the irreducible characters of $U_n$ which extend to irreducible characters of $U_n \langle \tau \rangle$ are exactly the $\tau$-invariant characters, which are the real-valued characters of $U_n$.  If $\chi$ is an irreducible real-valued character of $U_n$, then by \cite[Corollary 5.3]{ThVi07}, an irreducible character of $U_n \langle \tau \rangle$ which extends $\chi$ is real-valued exactly when $\chi$ is orthogonal, and takes non-real values whenever $\chi$ is symplectic.

As proven in \cite{GoVi08}, when $n$ is odd, there exists an element $x \tau$ of $U_n \langle \tau \rangle$ such that $(x\tau)^2$ is regular unipotent in $U_n$.  Part of the main result of \cite{GoVi08} is that if $\psi$ is an irreducible character of $U_n \langle \tau \rangle$ which extends a semisimple real-valued character of $U_n$, and $n$ is odd, then $\psi(x\tau) = \pm 1$.  In fact, in the case that $n$ is odd, every real-valued semisimple character of $U_n$ is orthogonal (see Section \ref{RegSemi}).

When $n$ is even, in \cite{GoVi08} it is shown that although $U_n \langle \tau \rangle$ has no elements which square to regular unipotent elements of $U_n$, there does exist an element $y \tau$ such that $(y\tau)^2 = -u$, where $u$ is regular unipotent in $U_n$.  In \cite[Theorem 6.4]{GoVi08}, it is proven that if $n=2m$ is even and $q$ is odd, and $\psi$ is an irreducible character of $U_n \langle \tau \rangle$ which is the extension of a semisimple symplectic character of $U_n$, then $\chi(y \tau)$ is a nonzero purely imaginary number, where $(y\tau)^2 = -u$ and $u$ is regular unipotent in $U_n$.  As a consequence, the element $y\tau$ is not a real element of $U_n \langle \tau \rangle$.  It is then conjectured that $\chi(y\tau)$ is equal to $\pm \sqrt{-q}$, and that there are exactly $q^{m-1}$ semisimple symplectic characters of $\U(2m, \fifisq)$.  Here, we prove the second part of this conjecture, in Theorem \ref{MainThm} below.  Counting the number of semisimple symplectic characters has significance in proving the first part of this conjecture, as it could be used in a calculation similar to that in the proof of \cite[Theorem 6.3]{GoVi08}.

As discussed in the Introduction of \cite{GoVi08}, there is a more general conjecture about the values of the characters of $U_n \langle \tau \rangle$ which are extended from irreducible real-valued characters of $U_n$, and their correspondence with characters extended from $\GL(n, \fifi)$.  In particular, let $G_n = \GL(n, \fifi)$, and let $G_n \langle \tau \rangle$ be the split extension of $G_n$ by the transpose-inverse automorphism.  When $n=2m$ is even and $q$ is odd, we expect there to be exactly $q^{m-1}$ semisimple real-valued characters of $G_n$ with the property that the extensions of these characters to $G_n \langle \tau \rangle$ take the values $\pm \sqrt{q}$ on elements $w \tau$, where $(w\tau)^2$ is the negative of a regular unipotent element of $G_n$.

In \cite{Go76}, Gow proves that the rational Schur index of any irreducible character of $\U(n, \fifisq)$ is at most $2$.  It is not known in general which characters of the finite unitary group have Schur index $1$ or $2$, although partial results are obtained in \cite{Go76, Ohm1, Ohm2}.  By the Brauer-Speiser Theorem, any symplectic character of a finite group has rational Schur index $2$, and so the main result of this paper also gives a lower bound for the number of irreducible characters of $\U(2m, \fifisq)$, $q$ odd, with Schur index $2$ (we could also add certain regular and unipotent characters to this count; see Section \ref{RegSemUni}).
\\
\\
\noindent {\bf Acknowledgements. }  The author thanks Fernando Rodriguez-Villegas for very helpful communication regarding Theorem \ref{CentralValues}.

\section{Preliminaries} \label{Pre}

\subsection{Frobenius-Schur indicators} \label{FS}

Let $G$ be a finite group, and let $(\pi, V)$ be an irreducible complex representation of $G$ with character $\chi$.  Define the {\em Frobenius-Schur indicator} of $\chi$ (or of $\pi$), denoted $\ep(\chi)$ (or $\ep(\pi)$), to be
$$ \ep(\chi) = \frac{1}{|G|} \sum_{g \in G} \chi(g^2).$$
By classical results of Frobenius and Schur, $\ep(\chi)$ always takes the value $0$, $1$, or $-1$, and $\ep(\chi) = 0$ if and only if $\chi$ is not a real-valued character.  Furthermore, if $\chi$ is real-valued, then $\ep(\chi) = 1$ if and only if $(\pi, V)$ is a real representation.  That is, $\ep(\chi) = 1$ if and only if there exists a basis for $V$ such that the matrix representation corresponding to $\pi$ with respect to this basis has image in the invertible matrices defined over $\real$.

If $\chi$ is a real-valued irreducible character of $G$, then $\chi$ is called an {\em orthogonal character} of $G$ if $\ep(\chi) = 1$, and $\chi$ is called a {\em symplectic character} of $G$ if $\ep(\chi) = -1$.

\subsection{The finite unitary group} \label{UnitaryDefn}

Let $q$ be a power of a prime $p$, let $\fifi$ denote a finite field with $q$ elements, and let $\fifibar$ be a fixed algebraic closure of $\fifi$.  Now let $\bG_n = \GL(n, \fifibar)$ be the group of invertible $n$-by-$n$ matrices with entries in $\fifibar$.  For $g = (a_{ij}) \in \bG_n$, let ${^T g}$ denote the transpose of $g$, so $^T g = (a_{ji})$.  Now, define the map $F$ on $\bG_n$ by
$$ F: g=(a_{ij}) \mapsto {^T (a_{ij}^q)^{-1}} = (a_{ji}^q)^{-1}.$$
The {\em finite unitary group}, denoted $\U(n,\fifisq)$, is defined to be the set of $F$-fixed points of $\bG_n$, so
$$ \U(n, \fifisq) = \{ g \in \GL(n, \fifibar) \; \mid \; F(g) = g \}.$$
We will also denote $\U(n,\fifisq)$ by $U_n$.

\subsection{Partitions and multi-partitions} \label{parts}
Let 
$$\cP=\bigcup_{n\geq 0} \cP_n,\qquad \text{where} \qquad \cP_n=\{\text{partitions of $n$}\}.$$
For $\nu = (\nu_1, \nu_2, \ldots, \nu_l) \in \cP_n$, $n \geq 1$, the {\em length} $\ell(\nu)$ of $\nu$ is the number of parts $\ell$ of $\nu$, and the {\em size} $|\nu|$ of $\nu$ is the sum of the parts of $\nu$, so $|\nu| = n$.  We consider $\cP_0$ to have one element, the empty partition $\varnothing$, which has no parts, so $\ell(\varnothing) = |\varnothing| = 0$.  Let $\nu'$ denote the conjugate of the partition $\nu$.  We also write
$$ \nu = (1^{m_1(\nu)} 2^{m_2(\nu)} \cdots ), \quad \text{ where } \quad m_i(\nu) = \big| \{ j \in \intg_{\geq 1} \mid \nu_j = i \} \big|.$$ 
Define $n(\nu)$ by $n(\nu) = \sum_j (j-1) \nu_j$.  

We may also think of a partition $\nu$ as collection of points, where the part $\nu_i$ corresponds to the points $(i,1), (i,2), \ldots, (i, \nu_i)$.  We correspond to this collection of points a {\em Young diagram} for $\nu$, which is a collection of rows of boxes in the positions given by the coordinates of that row.  We take the upper-left corner as $(1,1)$, and increase the first coordinate moving to the right, and the second coordinate moving downwards.  We write $\Box \in \nu$ to denote a box of the Young diagram of $\nu$.  If $\Box \in \nu$ is in position $(i,j)$, then the {\em hook length} of $\Box$ in $\nu$, $h(\Box)$, is defined to be
$$ h(\Box) = \nu_i - \nu'_j - i - j + 1.$$

Now let $\cX$ be some collection of finite sets (which in the next section we will take to be orbits of the Frobenius map).  An {\em $\cX$-partition} (or a {\em multi-partition}) is a function $\blam: \cX \rightarrow \cP$, where for $X \in \cX$, we write $\blam(X) = \blam^{(X)}$.  Define
$$ \cP^{\cX}_n = \left\{ \blam: \cX \rightarrow \cP \; \mid \; \sum_{X \in \cX} |X| \, |\blam^{(X)}| = n \right\}, \text{ and } \cP^{\cX} = \bigcup_{n \geq 0} \cP^{\cX}_n.$$
If $\blam \in \cP_n^{\cX}$, we say that $\blam$ has {\em size} $n$, and write $|\blam| = n$.  For any $\blam \in \cP^{\cX}$, define the conjugate $\blam' \in \cP^{\cX}$ by $(\blam')^{(x)} = (\blam^{(x)})'$, and define $n(\blam)$ and $\ell(\blam)$ by
$$ n(\blam) = \sum_{X \in \cX} |X| n(\blam^{(X)}), \quad \text{ and } \quad \ell(\blam) = \sum_{X \in \cX} \ell(\blam^{(X)}).$$
If $\blam \in \cP^{\cX}$, we may think of the Young diagram for $\blam$ to be a list of the Young diagrams $\blam^{(x)}$, labelled by each $x \in \cX$.  Then we write $\Box \in \blam$ for a box in one of the Young diagrams $\blam^{(x)}$.  If $\Box \in \blam$ is in position $(i,j)$ of $\blam^{(x)}$, for $x \in \cX$, then we define the {\em hook length} of $\Box \in \blam$, written ${\bf h}(\Box)$, to be ${\bf h}(\Box) = |x| h(\Box)$, where $h(\Box)$ is the hook length of $\Box \in \blam^{(x)}$.

\subsection{Symmetric functions} \label{symmfuns}

Let $X = \{X_1, X_2, \ldots \}$ be an infinite set of indeterminates.  Let $\Lambda = \Lambda(X)$ denote the ring of symmetric functions in $X$, a definition for which we refer to \cite[Section I.2]{Mac}.  For any $\nu = (\nu_1, \nu_2, \ldots, \nu_{\ell})$, we let $p_{\nu}(X)$ denote the {\em power-sum symmetric function}, defined by
$$ p_{\nu}(X) = p_{\nu_1}(X) p_{\nu_2}(X) \cdots p_{\nu_{\ell}}(X), \quad \text{where} \quad p_m(X) = \sum_{i \geq 1} X_i^m.$$
For any $\lambda \in \cP$, we let $s_{\lambda}(X)$ denote the {\em Schur function} (see \cite[Section I.3]{Mac}), and for any $\mu \in \cP$ and indeterminate $t$, we let $P_{\mu}(X ; t)$ denote the {\em Hall-Littlewood function} (see \cite[Section III.2]{Mac}).  Recall that the sets $\{ p_{\nu}(X) \; \mid \; \nu \in \cP \}$ and $\{s_{\lambda} \; \mid \; \lambda \in \cP \}$ are $\cplx$-bases for $\Lambda$, and for any $t \in \cplx^{\times}$, the set $\{ P_{\mu}(X;t) \; \mid \; \mu \in \cP \}$ is also a $\cplx$-basis for $\Lambda$.

\section{The characters of finite unitary groups}  \label{Char}

\subsection{Orbits of Frobenius} \label{Orbits}

Consider the Frobenius map $F$, defined in Section \ref{UnitaryDefn}, acting on $\bG_1 = \fifibar^{\times}$.  For any $\alpha \in \fifibar^{\times}$, let $[\alpha]$ denote the $F$-orbit of $\alpha$, so
$$ [\alpha] = \{ \alpha, \alpha^{-q}, \alpha^{q^2}, \ldots, \alpha^{(-q)^{d-1}} \},$$
where if $|[\alpha]|=d$, then $\alpha^{(-q)^d} = \alpha$.  Let $\Phi$ denote the collection of the orbits of $F$ acting on $\fifibar^{\times}$:
$$ \Phi = \{ F\text{-orbits of } \fifibar^{\times} \}.$$
We will typically denote an element of $\Phi$ by $f$, as these $F$-orbits correspond to certain polynomials over $\fifisq$ (see \cite{En62, ThVi07}).

Now let $T_m$ denote the $F^m$-fixed points of $\fifibar^{\times}$, so that, for example,
$$ T_1 = \{ \alpha \in \fifisq^{\times} \; \mid \; \alpha^{q+1} = 1 \}, $$
and when $m=2l$ is even, we have $T_{2l} = \mathbb{F}_{q^{2l}}^{\times}$.  Note that we have
$$ \bigcup_{m \geq 1} T_m = \fifibar^{\times}, $$
so that we may view $\fifibar^{\times}$ as the direct limit, 
$$ \bG_1 = \lim_{\rightarrow} T_m.$$
We also have norm maps, $N_{m,r}$, whenever $r|m$,
\begin{equation} \label{norm}
\begin{array}{rccl} N_{m,r}:& T_m & \longrightarrow & T_r\\ & \alpha & \mapsto & \prod_{i=0}^{(m/r) - 1} \alpha^{(-q)^{ri}}\end{array},\quad \text{where $m,r\in \intg_{\geq 1}$, $r\mid m$}.
\end{equation}

Now consider, for each $m \geq 1$, the group of characters $\widehat{T}_m$,
$$ \widehat{T}_m = \{ \xi: T_m \rightarrow \cplx^{\times} \; \mid \; \xi \text{ a multiplicative homomorphism} \}.$$
Whenever $r|m$, we may use the transpose of the norm map $N_{m,r}$, which we denote $N_{m,r}^{\star}$ to embed $\widehat{T}_r$ into $\widehat{T}_m$, as follows:
\begin{equation} \label{normstar}
\begin{array}{cccc}  N_{m,r}^{\star} :& \widehat{T}_r & \longrightarrow & \widehat{T}_m \\ & \xi & \mapsto & \xi \circ N_{m,r} \end{array}
\end{equation}
Now, we may consider the direct limit of the groups $\widehat{T}_m$ with respect to the maps $N_{m,r}^{\star}$, which we denote by $L$:
$$ L = \lim_{\rightarrow} \widehat{T}_m.$$
The Frobenius map $F$ acts naturally on each $\widehat{T}_m$, and so acts on their direct limit $L$.  Let $\Theta$ denote the collection of $F$-orbits on $L$:
$$ \Theta = \{ F\text{-orbits of } L\}.$$  
If $\xi \in L$, we let $[\xi]$ denote the $F$-orbit of $\xi$, so $[\xi] \in \Theta$.  

Note that we can also consider the inverse limit of the groups $T_m$ with respect to the norm maps $N_{m,r}$, which we denote by $T$:
$$T = \lim_{\leftarrow} T_m.$$
Then, the group of characters $\widehat{T}$ of $T$, is isomorphic to $L$.  If we consider the set $L_m$ of $F^m$-fixed elements of $L$, then we may identify it exactly with $\widehat{T}_m$.  With this identification, we may pair elements of $L$ with elements of $\fifibar^{\times}$, since
$$ \fifibar^{\times} = \bigcup_m T_m \quad \text{ and } \quad L = \bigcup_m L_m,$$
and $L_m$ is identified with $\widehat{T}_m$.  Given $\xi \in L$ and $a \in \fifibar^{\times}$, if $\xi \in L_m$ and $a \in T_m$, then we may consider the natural pairing $\xi(a)$, which we also denote by $\xi(a)_m$ to emphasize the location of the pairing.  If we have $\xi \in L_r$, and $r|m$, then we may embed $L_r$ in $L_m$ through the map $N_{m,r}^{\star}$.  In this case, if $a \in T_r$, then also $a \in T_m$, and we may take the pairing of $\xi$ with $a$ taking either $\xi \in L_r$ or $\xi \in L_m$.  Then, from the definitions of the maps $N_{m,r}$ in Equation (\ref{norm}) and $N_{m,r}^{\star}$ in Equation (\ref{normstar}), we have
\begin{equation} \label{pairing}
\xi(a)_m= \xi(N_{m,r}(a))_r = \xi(a)_r^{m/r}.
\end{equation}

When $q$ is odd, define the character $\sigma$ of $T_1$ by $\sigma(\beta) = -1$, where $\beta$ is a multiplicative generator of $T_1$, and also let $\sigma$ denote the corresponding element of $L$.  Note that the definition of $\sigma$ is independent of the choice of generator of $T_1$.  Also note that for any $m \geq 1$, if $\tilde{\beta}$ is a multiplicative generator of $T_m$, then we have $\sigma(\tilde{\beta})_m = -1$, since $N_{m,1}(\tilde{\beta})$ is a multiplicative generator for $T_1$.  

We will let ${\bf 1}$ denote the trivial character of any finite group, or of the element of $L$ corresponding to the trivial character.
\\
\\
\noindent {\bf Remark. }  The set $\Theta$ which we describe above was defined incorrectly in \cite{ThVi07}, where we erroneously identified $\fifibar^{\times}$ with the {\em inverse} limit of the $T_m$'s, rather than the {\em direct} limit, as above.  This error was carried over in \cite{GoVi08}.  This error is only cosmetic, however, in that it does not affect any of the results obtained in \cite{ThVi07, GoVi08}.  The description of $\Theta$ we give here should replace the flawed description given in \cite{ThVi07} and \cite{GoVi08}.

\subsection{Conjugacy classes} \label{Conj}

Recall the sets of $F$-orbits, $\Phi$ and $\Theta$, defined in Section \ref{Orbits}.  We consider $\cX$-partitions, where $\cX = \Phi$ or $\cX = \Theta$.  When $\cX = \Phi$, we have the following, which was proven in \cite{En62, Wa62}.

\begin{thm}[Ennola, Wall] \label{conj}
The conjugacy classes of the group $\U(n, \fifisq)$ are parametrized by $\cP_n^{\Phi}$, the $\Phi$-partitions of size $n$.
\end{thm}

If $\bmu \in \cP_n^{\Phi}$, we let $K^{\bmu}$ denote the conjugacy class of $\U(n,\fifisq)$ which corresponds to $\bmu$ by Theorem \ref{conj}.  We note that, in particular an element in the conjugacy class $K^{\bmu}$ has characteristic polynomial 
$$ \prod_{f \in \Phi} \prod_{\alpha \in f} (x-\alpha)^{|\bmu^{(f)}|},$$
(see \cite{En62,Wa62} for the precise correspondence).  For example, if $z = \alpha I$, $\alpha \in T_1$, is a central element of $\U(n, \fifisq)$, then $z \in K^{\bmu}$, where $\bmu^{(\{ \alpha \})} = (1^n)$, and $\bmu^{(f)} = \varnothing$ if $f \neq \{ \alpha \}$.  We also have a precise description of the orders of the centralizers of elements in each conjugacy class of $\U(n, \fifisq)$, as in the following.

\begin{thm}[Ennola, Wall] \label{cent}
Let $g \in \U(m, \fifisq)$ be in the conjugacy class $K^{\bmu}$, for $\bmu \in \cP^{\Phi}_m$.  Then the order of the centralizer of $g$ in $\U(m, \fifisq)$ is
$$ a_{\bmu} = (-1)^{|\bmu|} \prod_{f \in \Phi} a_{\bmu^{(f)}}( (-q)^{|f|} ), \text{ where } a_{\mu}(x) = x^{|\mu| + 2n(\mu)} \prod_i \prod_{j=1}^{m_i} ( 1 - (1/x)^j),$$
for $\mu = (1^{m_1} 2^{m_2} 3^{m_3} \cdots ) \in \cP$.
\end{thm}

We now consider the set $\cP_n^{\Theta}$, and we first describe a bijection between the sets $\cP_n^{\Theta}$ and $\cP_n^{\Phi}$ which will be used later.  Let $\blam \in \cP_n^{\Theta}$.  Then for each $\varphi \in \Theta$ such that $\blam^{(\varphi)} \neq \varnothing$, we must have $|\varphi| \leq n$.  Since an element $\xi \in \varphi$ satisfies $\xi \in \widehat{T}_{|\varphi|}$, and $\widehat{T}_r$ may be embedded in $\widehat{T}_m$ whenever $r|m$, as described in Section \ref{Orbits}, we may assume each $\varphi \in \Theta$ such that $\blam^{(\varphi)} \neq \varnothing$ satisfies $\varphi \subset \widehat{T}_{n!}$.  Similarly, if $\bmu \in \cP_n^{\Phi}$, then we may assume that each orbit $f \in \Phi$ such that $\bmu^{(f)} \neq \varnothing$ satisfies $f \subset T_{n!}$.

Both $T_{n!}$ and $\widehat{T}_{n!}$ are cyclic groups of order $q^{n!} - (-1)^{n!}$.  Let $\alpha$ be a multiplicative generator for $T_{n!}$, and define $\xi \in \widehat{T}_{n!}$ by letting $\xi(\alpha) = \zeta$, where $\zeta$ is a primitive $(n!)$-th root of unity.  Then, the map $\partial: T_{n!} \rightarrow \widehat{T}_{n!}$, defined by $\partial(\alpha) = \xi$, is an isomorphism, and the Frobenius map $F$ commutes with $\partial$.  So, $\partial$ may be extended to a map of $F$-orbits of $T_{n!}$ to $F$-orbits of $\widehat{T}_{n!}$, which we call $\tilde{\partial}$.  That is, for any $a \in T_{n!}$, we have $\tilde{\partial}([a]) = [\partial(a)]$.  Finally, we may use $\tilde{\partial}$ to define a bijection $\Delta: \cP_n^{\Theta} \rightarrow \cP_n^{\Phi}$, as follows.  For each $\blam \in \cP_n^{\Theta}$, define
\begin{equation} \label{bijection}
\Delta(\blam) = \bmu, \text{ where } \bmu^{([a])} = \blam^{(\tilde{\partial}([a]))} \text{ for each } a \in T_{n!}.
\end{equation}
It follows directly from the construction of $\Delta$ that it is a bijection.  We will see that in fact the set $\cP_n^{\Theta}$ naturally parametrizes the set of irreducible characters of the group $\U(n,\fifisq)$.

\subsection{The character ring and Deligne-Lusztig characters} \label{CharRing}

Let $C_n$ be the ring of $\cplx$-valued class functions of $\U(n,\fifisq)$, and let $C = \oplus_{n \geq 1} C_n$.  If $\psi_1 \in C_i$, $\psi_2 \in C_j$, for some $i, j \geq 1$, let $\psi_1 \odot \psi_2$ denote the class function obtained by Deligne-Lusztig induction (see \cite{Ca85, dellusz, dmbook} for a definition), so that
$$ \psi_1 \odot \psi_2 = R_{U_i \times U_j}^{U_{i+j}}(\psi_1 \otimes \psi_2) \in C_{i+j}.$$
Now, $C$ is a graded $\cplx$-algebra, where $\odot$ is the graded product (see \cite{ThVi07}).  We may equip each $C_n$ with the usual inner product of class functions of a finite group, denoted by $\langle \cdot, \cdot \rangle$, which can be extended to $C$ by defining $C_i$ and $C_j$ to be mutually orthogonal when $i \neq j$.

For each $\bmu \in \cP_n^{\Phi}$, $n \geq 1$, let $\kappa^{\bmu}$ denote the indicator class function for the conjugacy class $K^{\bmu}$ of $U_n$.  Note that the set $\{ \kappa^{\bmu} \mid \bmu \in \cP^{\Phi} \}$ is a $\cplx$-linear basis for $C$.

The maximal tori of $U_n$ are isomorphic to $T_{\nu} = T_{\nu_1} \times \cdots \times T_{\nu_{\ell}}$, as $\nu = (\nu_1, \ldots, \nu_{\ell})$ ranges over partitions $\nu \in \cP_n$.  If $T_{\nu}$ is a maximal torus of $U_n$, and $\theta: T_{\nu} \rightarrow \cplx^{\times}$ is a linear character of $T_{\nu}$, we let $R_{T_{\nu}, \theta}$ denote the {\em Deligne-Lusztig} virtual character of $U_n$ corresponding to the pair $(T_{\nu}, \theta)$ (see \cite{Ca85, dellusz, dmbook}).  Given a pair $(T_{\nu}, \theta)$, we have $\theta = \theta_1 \otimes \theta_2 \otimes \cdots \otimes \theta_{\ell}$, where $\theta_i \in \widehat{T}_{\nu_i}$.  We associate to the pair $(T_{\nu}, \theta)$ an element $\btau_{\nu, \theta} \in \cP_n^{\Theta}$ by defining, for each $\varphi \in \Theta$,
\begin{equation} \label{DLcorresp}
\btau_{\nu, \theta}^{(\varphi)} = (\nu_{i_1}/|\varphi|, \nu_{i_2}/|\varphi|, \ldots) \quad \text{such that} \quad \theta_{i_1}, \theta_{i_2}, \ldots \in \varphi.
\end{equation}
It follows from \cite{LuSr77} that the set $\{R_{T_{\nu}, \theta} \mid \btau_{\nu, \theta} \in \cP^{\Theta} \}$ is another $\cplx$-linear basis for $C$.

\subsection{The characteristic map}  \label{CharMap}

In this section, we give a description of the irreducible characters of $\U(n,\fifisq)$ using symmetric functions, following the development in \cite{ThVi07}.  

For each $f \in \Phi$ and each $\varphi \in \Theta$, define infinite sets of indeterminates 
$$X^{(f)} = \{X_i^{(f)} \; \mid \; i \geq 1 \} \; \text{ and } \; Y^{(\varphi)} = \{Y_i^{(\varphi)} \; \mid \; i \geq 1 \}.$$ 
We consider symmetric functions in the $X$ and $Y$ variables, and we relate symmetric functions in these variables through the following transform:
\begin{equation} \label{transform}
p_n(Y^{(\varphi)}) = (-1)^{n|\varphi| -1} \sum_{\alpha \in T_{n|\varphi|}} \xi(\alpha) p_{n|\varphi|/|f_{\alpha}|}(X^{(f_{\alpha})}),
\end{equation}
where $\alpha \in f_{\alpha}$, $\xi \in \varphi$, the pairing $\xi(\alpha)$ is taken with $\xi \in \widehat{T}_{n|\varphi|}$, and the power-sum symmetric function $p_z$ is taken to be $0$ if $z$ is not an integer.

For any $\bmu \in \cP^{\Phi}$, define a generalized version of the Hall-Littlewood symmetric function $P_{\bmu}$ in the $X$ variables by
$$ P_{\bmu} = q^{-n(\bmu)} \prod_{f \in \Phi} P_{\bmu^{(f)}} (X^{(f)}; q^{-|f|}).$$
We now define a ring of symmetric functions $\Lambda$ by
$$ \Lambda = \bigoplus_n \Lambda_n \quad \text{where} \quad \Lambda_n = \cplx \spanning \{P_{\bmu} \; \mid \; \bmu \in \cP^{\Phi}_n \}.$$
For any $\blam, \bnu \in \cP^{\Theta}$, we define a generalized Schur function $s_{\blam}$ and a power-sum function in the $Y$ variables by
$$ s_{\blam} = \prod_{\varphi \in \Theta} s_{\blam^{(\varphi)}}(Y^{(\varphi)}) \quad {\rm and } \quad p_{\bnu} = \prod_{\varphi \in \Theta} p_{\bnu^{(\varphi)}}(Y^{(\varphi)}).$$
It follows from Equation (\ref{transform}), and the results recalled in Section \ref{symmfuns} that we also have
$$ \Lambda = \bigoplus_n \cplx \spanning \{s_{\blam} \; \mid \; \blam \in \cP^{\Theta}_n \} = \bigoplus_n \cplx \spanning \{p_{\bnu} \; \mid \; \bnu \in \cP^{\Theta}_n \}.$$
Now, $\Lambda$ is a graded $\cplx$-algebra, with graded multiplication given by normal multiplication of symmetric functions, and the grading given by the degree of a symmetric function, except by weighting the variables in $X^{(f)}$ by $|f|$ and the variables in $Y^{(\varphi)}$ by $|\varphi|$.  That is, each element of $\Lambda_n$ has grading $n$.

We define a Hermitian inner product on $\Lambda$ by defining, for any $\bmu, \bnu \in \cP^{\Phi}$,
\begin{equation} \label{inner} 
\langle P_{\bmu}, P_{\bnu} \rangle = \delta_{\bmu \bnu} a_{\bmu}^{-1}, 
\end{equation}
and extending linearly, where $a_{\bmu}$ is the size of the centralizer of the conjugacy class $K_{\bmu}$ in $U_{\bmu}$, as given in Theorem \ref{cent}.  Now, both $C$ and $\Lambda$ are graded $\cplx$-algebras with Hermitian inner products.  For a proof of the following, see \cite{ThVi07} and the references listed there.

\begin{thm} \label{Characters}
Define a map $\ch: C \rightarrow \Lambda$ by letting $\ch(\kappa_{\bmu}) = P_{\bmu}$ for each $\bmu \in \cP^{\Phi}$, and extending linearly.  Then $\ch$ is an isometric isomorphism of graded $\cplx$-algebras.

For any $\blam \in \cP^{\Theta}$, define $\chi^{\blam} \in C$ by
$$ \chi^{\blam} = \ch^{-1}\left((-1)^{\lfloor |\blam|/2 \rfloor + n(\blam)} s_{\blam}\right).$$
Then, the set $\{ \chi^{\blam} \mid \blam \in \cP_m^{\Theta} \}$ is the set of irreducible characters of $\U(m, \fifisq)$. 

For any $\bnu \in \cP^{\Theta}$, define $R_{\bnu} \in C$ by
$$ R_{\bnu} = \ch^{-1} \left((-1)^{|\bnu| + \ell(\bnu)} p_{\bnu} \right).$$
Then, the set $\{ R_{\bnu} \mid \bnu \in \cP_n^{\Theta} \}$ is the set of Deligne-Lusztig virtual characters of $\U(n, \fifisq)$, and for each $\bnu \in \cP^{\Theta}$, $R_{\bnu} = R_{T_{\eta}, \theta}$ where $\bnu = \btau_{\eta, \theta}$.
\end{thm}

If $\chi^{\blam}$ is a character of $\U(m, \fifisq)$, and $\bmu \in \cP_m^{\Phi}$ so that $K^{\bmu}$ is a conjugacy class of $\U(m, \fifisq)$, then let $\chi^{\blam}(\bmu)$ denote the value of the character $\chi^{\blam}$ at any element of the conjugacy class $K^{\bmu}$.  Then Theorem \ref{Characters} implies that
$$
(-1)^{\lfloor m/2 \rfloor + n(\blam)} s_{\blam} = \sum_{\bmu \in \cP_m^{\Theta}} \chi^{\blam}(\bmu) P_{\bmu},$$
and so, from the definition of the inner product on $\Lambda$ in (\ref{inner}), we have 
$$ \chi^{\blam}(\bmu) = \langle (-1)^{\lfloor m/2 \rfloor + n(\blam)} s_{\blam}, a_{\bmu} P_{\bmu} \rangle.$$
It also follows from Theorem \ref{Characters} that the $s_{\blam}$'s form an orthonormal basis for $\Lambda$, and so we have
\begin{equation} \label{charvalues}
a_{\bmu} P_{\bmu} = \sum_{\blam \in \cP_m^{\Theta}} (-1)^{\lfloor m/2 \rfloor + n(\blam)} \chi^{\blam}(\bmu) \bar{s}_{\blam},
\end{equation}
where $\bar{s}_{\blam}$ is obtained by taking the complex conjugates of the coefficients of the $P_{\bmu}$'s, when $s_{\blam}$ is expanded in terms of the $P_{\bmu}$.

\subsection{Character degrees and real-valued characters}

Recall the definition of the hook length of a box in the Young diagram of a multi-partition, given in Section \ref{parts}.  Let $1$ denote the identity element in the group $U_m = \U(m,\fifisq)$.  If $\chi^{\blam}$ is an irreducible character of $U_m$, then the {\em degree} of $\chi^{\blam}$ is the value $\chi^{\blam}(1)$.  The following result, proven in \cite[Theorem 5.1]{ThVi07}, gives a formula for the degrees of the characters of $U_m$.

\begin{thm} \label{degrees} Let $\blam \in \cP_m^{\Theta}$.  The degree of the character $\chi^{\blam}$ is given by
$$ \chi^{\blam}(1) = q^{n(\blam')} \frac{\prod_{1 \leq i \leq m} (q^i - (-1)^i)}{\prod_{\Box \in \blam} (q^{{\bf h}(\Box)} - (-1)^{{\bf h}(\Box)})}.$$
\end{thm}

We will give a generalization of Theorem \ref{degrees} in Section \ref{CentralCharacters}, where we give a formula for the value of $\chi^{\blam}$ at any element of the center of the group $U_m$.

Let $\xi \in L$, where $L$ is the direct limit of the character groups $\widehat{T}_m$ defined in Section \ref{Orbits}.  If $\varphi = [\xi]$ is the $F$-orbit of $\xi$, then we let $\bar{\varphi}$ denote the $F$-orbit of the complex conjugate $\bar{\xi}$, so $\bar{\varphi} = [\bar{\xi}]$.  For any $\blam \in \cP_n^{\Theta}$, we define $\bar{\blam}$ by $\bar{\blam}^{(\varphi)} = \blam^{(\bar{\varphi})}$ for each $\varphi \in \Theta$.  The following result, proven in \cite[Lemma 3.1(ii)]{GoVi08}, describes exactly when a character of $\U(n, \fifisq)$ is real-valued.

\begin{thm} \label{RealChars} Let $\blam \in \cP_n^{\Theta}$.  Then the character $\chi^{\blam}$ of ${\rm U}(n, \fifisq)$ is real-valued if and only if $\bar{\blam} = \blam$.
\end{thm}

\section{Values of central characters} \label{CentralCharacters}

In this section, we obtain a formula for the value of any irreducible character of $\U(n, \fifisq)$ at any element of the center of this group, which is the set of scalar matrices $\{ \alpha I \mid \alpha \in T_1 \}$.  The formula we obtain is analogous to the formula for $\GL(n, \fifi)$ obtained by Macdonald in \cite[IV.6, Example 2]{Mac}, and our proof of Theorem \ref{CentralValues} below follows the proof for $\GL(n, \fifi)$ which is suggested in \cite{Mac}.  The proof is also similar to \cite[Theorem 5.1]{ThVi07}.

\begin{thm} \label{CentralValues} Let $z = \alpha I$ be in the center of ${\rm U}(n, \fifisq)$, where $\alpha \in T_1$, and let $\chi^{\blam}$ be an irreducible character of ${\rm U}(n, \fifisq)$.  Define $\omega_{\blam} \in L_1$ by
$$ \omega_{\blam} = \prod_{\varphi \in \Theta} \prod_{\xi \in \varphi} \xi^{|\blam^{(\varphi)}|}.$$
Then, $\chi^{\blam}(z) = \omega_{\blam}(\alpha)_1 \chi^{\blam}(1)$.
\end{thm}
\begin{proof} Let $z = \alpha I$ be an element of the center of $\U(m, \fifisq)$, with $\alpha \in T_1$.  The conjugacy class $\{ z \}$ corresponds to the element $\bmu_{\alpha} \in \cP_m^{\Phi}$ defined by
$$ \bmu_{\alpha}^{(\{\alpha\})} = (1^m), \quad \text{ and } \quad \bmu_{\alpha}^{(f)} = \varnothing \quad \text{ if } \quad \{ \alpha \} \neq f \in \Phi.$$
From (\ref{charvalues}), we have
\begin{equation} \label{value}
a_{\bmu_{\alpha}} P_{\bmu_{\alpha}} = \sum_{\blam \in \cP_m^{\Theta}} (-1)^{\lfloor m/2 \rfloor + n(\blam)} \chi^{\blam}(z) \bar{s}_{\blam}.
\end{equation}
From Theorem \ref{cent} and the definition of $\bmu_{\alpha}$, we have
$$\begin{array}{rcl}
a_{\bmu_{\alpha}} = (-1)^m a_{(1^m)}(-q) & = & (-1)^m (-q)^{m + 2n((1^m))} \prod_{j=1}^m \left( 1 - (-q)^{-j} \right) \\
 & = & (-1)^m (-q)^{m(m+1)/2} \prod_{j=1}^m \left( 1 - (-q)^{-j} \right) \\
 & = & \prod_{j=1}^m \left(1 - (-q)^j \right).
\end{array}$$
Now, since $P_{(1^m)}(x ; t) = e_m(x)$, the elementary symmetric function (by \cite[III.2.8]{Mac}), we have
\begin{equation} \label{aPfactors}
\begin{array}{rcl}
a_{\bmu_{\alpha}} P_{\bmu_{\alpha}} & = & \prod_{j=1}^m \left( 1 - (-q)^{j} \right) P_{(1^m)}\left(X^{(\{\alpha\})}; (-q)^{-1}\right) \\
 & = & \prod_{j=1}^m \left(1 - (-q)^j \right) e_m(X^{(\{ \alpha \})}).
\end{array}
\end{equation}
So, from (\ref{value}) and (\ref{aPfactors}), we have
\begin{equation} \label{value2}
\prod_{j=1}^m \left(1 - (-q)^j \right) e_m(X^{(\{ \alpha \})}) = \sum_{\blam \in \cP_m^{\Theta}} (-1)^{\lfloor m/2 \rfloor + n(\blam)} \chi^{\blam}(z) \bar{s}_{\blam}.
\end{equation}

Define a $\cplx$-algebra homomorphism $\delta_{\alpha}: \Lambda \rightarrow \cplx$ by
$$ \delta_{\alpha}(p_m(X^{(f)})) = \left\{ \begin{array}{ll} (-1)^{m-1}/(q^m - (-1)^m) & \text{ if } f = \{ \alpha \}, \\ 0 & \text{ if } f \neq \{ \alpha \}. \end{array} \right.$$
Now, as in \cite[Equation (5.4)]{ThVi07}, we have
\begin{equation} \label{logsum}
\log \left( \sum_{\blam \in \cP^{\Theta}} s_{\blam} \otimes \bar{s}_{\blam} \right) = \sum_{n \geq 1} \frac{1}{n} \sum_{f \in \Phi} (q^{n|f|} - (-1)^{n|f|}) p_n(X^{(f)}) \otimes p_n(X^{(f)}).
\end{equation}
Noting that $\delta_{\alpha}$ passes through the logarithm function on symmetric functions, we apply $\delta_{\alpha} \otimes 1$ to both sides of (\ref{logsum}) and obtain
$$\log \left( \sum_{\blam \in \cP^{\Theta}} \delta_{\alpha}(s_{\blam}) \bar{s}_{\blam} \right) = \sum_{m \geq 1} \frac{(-1)^{m-1}}{m} p_m(X^{( \{\alpha\} )}) = \log \prod_i (1 + X_i^{( \{ \alpha \} )}).$$
Applying the exponential map, and expanding the product on the right into a sum of elementary symmetric functions, we have
$$ \sum_{m \geq 0} e_m(X^{( \{ \alpha \} )}) = \sum_{ \blam \in \cP^{\Theta}} \delta_{\alpha}(s_{\blam}) \bar{s}_{\blam}, \; \text{ and so } \; e_m(X^{( \{ \alpha \} )}) = \sum_{\blam \in \cP_m^{\Theta}} \delta_{\alpha}(s_{\blam}) \bar{s}_{\blam}. $$
Now, by comparing the coefficients of $\bar{s}_{\blam}$ in this with those in (\ref{value2}), we have
\begin{equation} \label{value3}
\chi^{\blam}(z) = (-1)^{\lfloor m/2 \rfloor + n(\blam)} \prod_{j=1}^m \left(1 - (-q)^j \right) \delta_{\alpha} (s_{\blam}).
\end{equation}

Using the transform in Equation (\ref{transform}), we calculate $\delta_{\alpha}(p_m(Y^{(\varphi)}))$, for any $\varphi \in \Theta$, and obtain
$$\delta_{\alpha}(p_m(Y^{(\varphi)})) = \frac{\xi(\alpha)_{m|\varphi|}}{q^{m|\varphi|}- (-1)^{m|\varphi|}},$$
where $\xi \in \varphi$, and since $\alpha \in T_1$, the value is invariant under the choice of $\xi$.  Since we also have $\xi \in T_{|\varphi|}$, we may apply Equation (\ref{pairing}), and obtain
$$\delta_{\alpha}(p_m(Y^{(\varphi)})) = (-1)^{|\varphi|m} \frac{\xi(\alpha)^m_{|\varphi|}}{(-q)^{|\varphi|m} - 1} = (-1)^{|\varphi|m} \sum_{i \geq 1} \xi(\alpha)^m_{|\varphi|} (-q)^{-i |\varphi| m}.$$
So, we have
$$ \delta_{\alpha}(p_m(Y^{(\varphi)})) = (-1)^{|\varphi|m} p_m \Big( \xi(\alpha)_{|\varphi|} (-q)^{-|\varphi|}, \xi(\alpha)_{|\varphi|} (-q)^{-2|\varphi|}, \ldots \Big).$$
That is, when $\delta_{\alpha}$ is applied to a homogeneous symmetric function in $Y^{(\varphi)}$ of graded degree $|\varphi|m$, the effect is multiplication by $(-1)^{|\varphi|m}$ and replacing the variable $Y_i^{(\varphi)}$ by $\xi(\alpha)_{|\varphi|} (-q)^{-i|\varphi|}$.  In particular, for any $\lambda \in \cP$, $\varphi \in \Theta$, we have
$$\begin{array}{rcl}
\delta_{\alpha}\left(s_{\lambda}(Y^{(\varphi)}) \right)  & =  &(-1)^{|\varphi| \, |\lambda|} s_{\lambda} \Big( \xi(\alpha)_{|\varphi|} (-q)^{-|\varphi|}, \xi(\alpha)_{|\varphi|} (-q)^{-2|\varphi|}, \ldots \Big)\\
 & = & q^{-|\varphi| \, |\lambda|} \xi(\alpha)_{|\varphi|}^{|\lambda|} s_{\lambda} \Big(1,  (-q)^{-|\varphi|}, (-q)^{-2|\varphi|}, \ldots \Big).
\end{array}$$
Now, applying \cite[I.3, Example 2]{Mac}, we may rewrite the above as
\begin{equation} \label{delta1}
\delta_{\alpha}\left(s_{\lambda}(Y^{(\varphi)})\right) = \xi(\alpha)_{|\varphi|}^{|\lambda|} \frac{(-1)^{|\varphi| |\lambda|} (-q)^{-|\varphi| (|\lambda| + n(\lambda))}}{\prod_{\Box \in \lambda} \left( 1 - (-q)^{-|\varphi| h(\Box)} \right)}.
\end{equation}
As mentioned above, regardless of the choice of $\tilde{\xi} \in \varphi$, the value $\tilde{\xi}(\alpha)_{|\varphi|}$ is unchanged.  So, for any $\tilde{\xi} \in \varphi$, we have
$$ \tilde{\xi}(\alpha)_{|\varphi|}^{|\lambda|}= \left( \prod_{\xi \in \varphi} \xi(\alpha)_{|\varphi|}^{|\lambda|} \right)^{1/|\varphi|}.$$
Now, we have $\prod_{\xi \in \varphi} \xi^{|\lambda|}$ is invariant under $F$-action, and so it is in $L_1$.  So, we may evaluate the above character value with the point of view that $\prod_{\xi \in \varphi} \xi^{|\lambda|} \in L_1$, and applying (\ref{pairing}), we have
\begin{equation} \label{xiprod}
\tilde{\xi}(\alpha)_{|\varphi|}^{|\lambda|}= \left( \prod_{\xi \in \varphi} \xi(\alpha)^{|\lambda||\varphi|}_1 \right)^{1/|\varphi|} = \prod_{\xi \in \varphi} \xi(\alpha)^{|\lambda|}_1.
\end{equation}
So, in Equation (\ref{delta1}), we may replace $\xi(\alpha)_{|\varphi|}^{|\lambda|}$ by $\prod_{\xi \in \varphi} \xi(\alpha)_1^{|\lambda|}$.  

Since $s_{\blam} = \prod_{\varphi \in \Theta} s_{\blam^{(\varphi)}}(Y^{(\varphi)})$, then from (\ref{delta1}) and (\ref{xiprod}), we now have
\begin{equation} \label{deltavalue}
\delta_{\alpha}(s_{\blam}) = 
\left( \prod_{\varphi \in \Theta} \prod_{\xi \in \varphi} \xi^{|\blam^{(\varphi)}|}(\alpha)_1 \right) \left( \frac{(-1)^{|\blam|} (-q)^{-|\blam| - n(\blam)}}{\prod_{\Box \in \blam} (1 - (-q)^{-{\bf h}(\Box)})} \right).
\end{equation}
Finally, we may substitute (\ref{deltavalue}) into (\ref{value3}) to compute the value of $\chi^{\blam}(z)$.  The first factor of (\ref{deltavalue}) is exactly $\omega_{\blam}(\alpha)_1$.  Now consider the second factor of the product in (\ref{deltavalue}), together with the other factors on the right-hand side of (\ref{value3}) other than $\delta_{\alpha}(s_{\blam})$.  It follows from Theorem \ref{degrees}, together with the calculation made at the end of the proof of Theorem \ref{degrees} as it appears in \cite[Theorem 5.1]{ThVi07}, that  
\begin{equation} \label{degreecalc}
\chi^{\blam}(1) = (-1)^{\lfloor m/2 \rfloor + n(\blam)} (-1)^{|\blam|} (-q)^{-|\blam| - n(\blam)} \frac{\prod_{j=1}^m \left(1 - (-q)^j \right)}{\prod_{\Box \in \blam} (1 - (-q)^{-{\bf h}(\Box)})}.
\end{equation}
Combining (\ref{value3}) with (\ref{deltavalue}) and (\ref{degreecalc}), we finally have $\chi^{\blam}(z) = \omega_{\blam}(\alpha)_1 \chi^{\blam}(1)$, as desired.
\end{proof}

For an irreducible representation $\pi$ of a finite group $G$ with character $\chi$, it follows from Schur's Lemma that the center of $G$ acts by scalars on the space corresponding to $\pi$, yielding the {\em central character} of the representation $\pi$ (or of the character $\chi$), which we denote $\omega_{\chi}$.  If $G = \U(n, \fifisq)$, and $\chi = \chi^{\blam}$, then Theorem \ref{CentralValues} may also be stated as $\omega_{\chi}(z) = \omega_{\blam}(\alpha)_1$, for any element $z = \alpha I$ of the center of $\U(n, \fifisq)$.

\section{Regular, semisimple, and unipotent characters}  \label{RegSemUni}

\subsection{Regular and semisimple characters} \label{RegSemi}

If $G$ is a finite group of Lie type defined over $\fifi$, and $q$ is the power of the prime $p$, where $p$ is a {\em good prime} for $G$ (see \cite{Ca85}), then a character of $G$ is called {\em semisimple} when its degree is not divisible by $p$.  In the case that $G = \U(n,\fifisq)$, $p$ is always a good prime for $G$ (since the group is type $A$).  From Theorem \ref{degrees}, the $p$-part of the degree of the character $\chi^{\blam}$ of $U_m$, where $\blam \in \cP_m^{\Theta}$, is exactly $q^{n(\blam')}$.  Thus, the character $\chi^{\blam}$ is semisimple exactly when $n(\blam') = 0$.  Since 
$$ n(\blam') = \sum_{\varphi \in \Theta} |\varphi| n((\blam')^{(\varphi)}), \quad \text{where} \quad n(\lambda') = \sum_{i \geq 1} (i-1)\lambda'_i \quad \text{for} \quad \lambda \in \cP,$$ 
it follows that $n(\blam') = 0$ if and only if $\ell((\blam')^{(\varphi)}) \leq 1$ for every $\varphi \in \Theta$.  So,
\begin{equation} \label{semisimple}
\chi^{\blam} \text{ is semisimple} \Leftrightarrow \text{for all } \varphi \in \Theta, \; \blam^{(\varphi)} = (1^{m_{\varphi}}) \; \text{for some } m_{\varphi} \geq 0.
\end{equation} 

Recall that the {\em Gelfand-Graev character} $\Gamma$, of a finite group of Lie type $G$, is obtained by inducing a non-degenerate linear character of the Sylow $p$-subgroup of $G$ up to $G$ (see \cite[Section 8.1]{Ca85}).  An irreducible character of $G$ is {\em regular} if it is a component of the Gelfand-Graev character $\Gamma$.  When $G = \U(n, \fifisq)$, it follows from results in \cite{dellusz} that for an irreducible character $\chi^{\blam}$ of $U_n$,
$$ \chi^{\blam} \text{ is regular} \Leftrightarrow \ell(\blam^{(\varphi)}) \leq 1 \text{ for every } \varphi \in \Theta.$$
The following result, proven in \cite{Vi09}, gives the Frobenius-Schur indicators of the real-valued regular and semisimple characters of the finite unitary groups.  We note that Theorem \ref{FScentral}(1) below also follows from \cite[Theorem 7(ii)]{Ohm2}.

\begin{thm} \label{FScentral} Let $\chi$ be a real-valued semisimple or regular character of ${\rm U}(n, \fifisq)$.  Then:
\begin{enumerate}
\item If $n$ is odd or $q$ is even, then $\ep(\chi) = 1$.
\item If $n$ is even and $q$ is odd, then $\ep(\chi) = \omega_{\chi}(\beta I)$, where $\beta$ is a multiplicative generator for $T_1$, and $\omega_{\chi}$ is the central character corresponding to $\chi$.
\end{enumerate}
\end{thm}

So, $\U(n, \fifisq)$ only has semisimple (or regular) symplectic characters when $n$ is even and $q$ is odd.  In the main result, we count the number of semisimple symplectic characters of $\U(n, \fifisq)$, which is the same as the number of regular symplectic characters of $\U(n, \fifisq)$.  In fact, as proven in \cite[Corollary 3.1]{Vi09}, if $G$ is a finite group of Lie type defined over $\fifi$, which comes from a connected reductive group over $\fifibar$ with connected center, then the number of semisimple symplectic (respectively, orthogonal) characters of $G$ is equal to the number of regular symplectic (respectively, orthogonal) characters of $G$.

\subsection{Unipotent characters} \label{Uni}

A {\em unipotent character} of $\U(n, \fifisq)$ is an irreducible character which appears as a component of a Deligne-Lusztig character of the form $R_{T_{\nu}, {\bf 1}}$, when it is decomposed into a linear combination of irreducible characters.  In this section, we use unipotent characters to show that the results in Theorem \ref{FScentral} on Frobenius-Schur indicators do not hold in general for all characters of $\U(n, \fifisq)$, expanding on the remark made immediately after the proof of \cite[Theorem 5.2]{Vi09}. 

Consider a Deligne-Lusztig character of the form $R_{T_{\nu}, {\bf 1}}$, where $\nu \in \cP_n$.  From Equation (\ref{DLcorresp}), we have 
$$ \btau_{\nu, {\bf 1}}^{(\varphi)} = \left\{ \begin{array}{ll} \nu & \text{ if } \varphi = \{ {\bf 1} \}, \\ \varnothing & \text{ otherwise.} \end{array} \right. $$
From Theorem \ref{Characters}, we have
\begin{equation} \label{DLcharmap}
\ch(R_{T_{\nu}, {\bf 1}}) = (-1)^{n + \ell(\nu)} p_{\nu}(Y^{( \{ {\bf 1} \} )}).
\end{equation}
Also from Theorem \ref{Characters}, to expand the Deligne-Lusztig character $R_{T_{\nu}, {\bf 1}}$ as a linear combination of irreducible characters of $\U(n, \fifisq)$ is equivalent to expanding the right-hand side of (\ref{DLcharmap}) into a linear combination of Schur functions.  We can write a power-sum symmetric function as a linear combination of Schur functions using the characters of the symmetric group, as in \cite[I.7.8]{Mac}, and doing so for (\ref{DLcharmap}), we obtain Schur functions of the form $s_{\lambda}(Y^{( \{ {\bf 1} \} )})$, where $\lambda \in \cP_n$.  As we vary $\nu \in \cP_n$, we will obtain a Schur function occurring with nonzero coefficient for every $\lambda \in \cP_n$, since by \cite[I.7.8]{Mac}, this is equivalent to the fact that for every conjugacy class of the symmetric group $S_n$, there is an irreducible character of $S_n$ which is nonzero on that conjugacy class.

It follows that the unipotent characters of $\U(n, \fifisq)$ are exactly of the form $\chi^{\blam}$, where $\blam^{(\varphi)} = \varnothing$ when $\varphi \neq \{ {\bf 1} \}$.  In this way, the unipotent characters are parametrized by partitions $\lambda \in \cP_n$, and we write $\chi^{\lambda}$ for the unipotent character of the form $\chi^{\blam}$, where $\blam^{ ( \{ {\bf 1} \} ) } = \lambda$.

Ohmori \cite{Ohm2} has completely determined the Frobenius-Schur indicators of the unipotent characters of $\U(n, \fifisq)$, and it is this result which we will apply.  We first prove a lemma which allows us to write Ohmori's result in a different form.  Given a partition $\lambda \in \cP_n$, the {\em $2$-core} of $\lambda$ is the partition $\tilde{\lambda} \subset \lambda$ of smallest size contained in $\lambda$, such that the skew partition $\lambda/ \tilde{\lambda}$ may be tiled by dominoes.  Equivalently, the $2$-core of $\lambda$ is the partition which remains after removing all possible {\em rim $2$-hooks} from the diagram for $\lambda$ (see \cite[Section 2.7]{JaKe81}).  The nonempty $2$-cores are exactly the stairstep partitions, which are partitions of the form $(k, k-1, \ldots, 2, 1)$.  Let $c_2(\lambda)$ denote the $2$-core of $\lambda$.  We have the following.

\begin{lemma} \label{2corelemma} Let $\lambda$ be a partition of $n$.  Let ${\rm ohl}(\lambda)$ and ${\rm ehl}(\lambda)$ denote the number of odd and even hook-lengths of $\lambda$, respectively.  Then,
$$ {\rm ohl}(\lambda) - {\rm ehl}(\lambda) = |c_2(\lambda)|.$$
\end{lemma}
\begin{proof}  We note that the claim is equivalent to 
\begin{equation} \label{claim}
|\lambda| = |c_2(\lambda)| + 2({\rm ehl}(\lambda)),
\end{equation} 
since ${\rm ohl}(\lambda) + {\rm ehl}(\lambda) = |\lambda|$.  The {\em $2$-weight} of $\lambda$, which we denote $w_2(\lambda)$, is the number of rim $2$-hooks which have to be removed from $\lambda$ in order to get the $2$-core of $\lambda$ (see \cite[Section 2.7]{JaKe81}).  So, $|\lambda| = |c_2(\lambda)| + 2w_2(\lambda)$.  From \cite[2.7.40]{JaKe81}, we have $w_2(\lambda) = {\rm ehl}(\lambda)$, giving (\ref{claim}).
\end{proof}

Using Lemma \ref{2corelemma}, we may now state the result of Ohmori \cite[Corollary]{Ohm2} in the following form.

\begin{thm}[Ohmori] \label{ohmori} Let $\chi^{\lambda}$ be a unipotent character of ${\rm U}(n, \fifisq)$ corresponding to the partition $\lambda$ of $n$.  Then,
$$ \ep(\chi^{\lambda}) = (-1)^{\lfloor |c_2(\lambda)|/2 \rfloor}.$$
\end{thm}

We now use Theorem \ref{ohmori} to see that the results in Theorem \ref{FScentral} do not extend in general to all real-valued characters of ${\rm U}(n, \fifisq)$.  Note that from Theorem \ref{CentralValues}, if $\chi^{\lambda}$ is a unipotent character of ${\rm U}(n, \fifisq)$, then the central character $\omega_{\chi^{\lambda}}$ of $\chi^{\lambda}$ is trivial, so $\omega_{\chi^{\lambda}}(z) = 1$ for all central elements $z$ of ${\rm U}(n, \fifisq)$.  

Consider the case when $n$ is even and $q$ is odd, and let $n = 2m \geq 6$.  Consider the partition $\lambda = (n-3, 2, 1)$.  Then, the $2$-core of $\lambda$ is obtained by removing $n-6$ squares from the first row of $\lambda$, so that $c_2(\lambda) = (3,2,1)$.  So, by Theorem \ref{ohmori}, $\ep(\chi^{\lambda}) = -1$, while the central character of $\chi^{\lambda}$ takes the value $1$ for every central element.  In particular, $\ep(\chi^{\lambda}) \neq \omega_{\chi^{\lambda}} (\beta I)$, where $\beta$ is a multiplicative generator of $T_1$.  Similarly, in the case that $n \geq 3$ is odd, if $\lambda = (n-1, 1)$, then $\chi^{\lambda}$ of $\U(n, \fifisq)$ is such that $\ep(\chi^{\lambda}) = -1$.

\section{Self-dual polynomials} \label{SelfDual}

Let $K$ be any field such that ${\rm char}(K) \neq 2$, and fix an algebraic closure $\bar{K}$ of $K$.  A polynomial $g(x) \in K[x]$ is called a {\em self-dual polynomial} if it is monic, non-constant, has non-zero constant term, and has the property that for any $\alpha \in \bar{K}^{\times}$, $\alpha$ is a root of $g(x)$ with multiplicity $m$ if and only if $\alpha^{-1}$ is a root of $g(x)$ with multiplicity $m$.  See \cite[Section 1]{Wo66} for the basic results on self-dual polynomials which we now state.  If $h(x) \in K[x]$ is a monic, non-constant polynomial with nonzero constant term, define $\tilde{h}(x) \in K[x]$ by
\begin{equation} \label{DualPoly}
\tilde{h}(x) = h(0)^{-1} x^d h(1/x), \quad \text{ where } \quad d ={\rm deg}(h(x)).
\end{equation}
Then $h(x)$ is a self-dual polynomial if and only if $h(x) = \tilde{h}(x)$.  Note that the only irreducible self-dual polynomials of odd degree are $x+1$ and $x-1$.  Any self-dual polynomial $g(x) = \tilde{g}(x)$ may be written as a product
\begin{equation} \label{DualProd}
g(x) = (x-1)^s (x+1)^t \prod_{i=1}^k (v_i(x) \tilde{v}_i(x))^{n_i} \prod_{j=1}^l r_j(x)^{m_j},
\end{equation}
where each $v_i(x), r_j(x) \in K[x]$ is irreducible over $K$, $v_i(x) \neq \tilde{v}_i(x)$, $r_j(x) = \tilde{r}_j(x)$, and $r_j(x) \neq x \pm 1$.  Note also that any self-dual $g(x) \in K[x]$ has constant term $\pm 1$, and the constant term is $1$ exactly when $s$ in the product (\ref{DualProd}) is even.

We now consider the case when $K= \fifi$, where $q$ is odd.  For any $\alpha \in \fifibar^{\times}$, if $d$ is the smallest non-negative integer such that $\alpha^{q^d} = \alpha$, then the polynomial
$$ h(x) = (x- \alpha)(x - \alpha^q)(x- \alpha^{q^2}) \cdots (x - \alpha^{q^{d-1}}),$$
is an irreducible polynomial in $\fifi[x]$, and every non-constant irreducible polynomial in $\fifi[x]$ with non-zero constant term is of this form.  Note that $h(x)$ is self-dual if and only if we have $\{ \alpha, \cdots, \alpha^{q^{d-1}} \} = \{ \alpha^{-1}, \cdots, \alpha^{-q^{d-1}} \}$, and if these sets are not equal, then we have
$$ h(x) \tilde{h}(x) = \prod_{i = 0}^{d-1} (x - \alpha^{q^i})(x - \alpha^{-q^{i}}),$$
is a self-dual polynomial.  The factorization in (\ref{DualProd}) says that these are essentially the only ways to get self-dual factors of self-dual polynomials.  We will say that these two types of polynomials are {\em irreducibly self-dual polynomials}, so that (\ref{DualProd}) gives the factorization of any self-dual polynomial into a product of irreducibly self-dual polynomials. 

Let $\alpha \in \fifibar^{\times}$, and let $d$ be the smallest non-negative integer such that $\alpha^{q^d} = \alpha$.  Note that we have
\begin{equation} \label{orbits}
[\alpha] \cup [\alpha^{-1}] = \{ \alpha, \ldots, \alpha^{q^{d-1}} \} \cup \{ \alpha^{-1}, \ldots, \alpha^{-q^{d-1}} \},
\end{equation}
where $[\alpha]$ is the $F$-orbit of $\alpha$, as in Section \ref{Orbits}.  It follows that every irreducibly self-dual polynomial in $\fifi[x]$ is of the form
$$ \prod_{ \gamma \in [\alpha] \cup [\alpha^{-1}]} (x - \gamma),$$
where $[\alpha] \in \Phi$.  From this observation, and the factorization in (\ref{DualProd}), it follows that self-dual polynomials in $\fifi[x]$ of degree $n$ are in one-to-one correspondence with elements $\bmu \in \cP^{\Phi}_n$ such that, for every $f \in \Phi$, $\bmu^{(f)} = (1^{m_f})$ for some $m_f \geq 0$, and $\bmu^{(f)} = \bmu^{(\bar{f})}$, where if $f = [\alpha]$, then $\bar{f} = [\alpha^{-1}]$.  Let $\cS_{\Phi, n}$ denote this subset of $\cP_n^{\Phi}$.  We define a bijection $\rho$ from $\cS_{\Phi, n}$ to the set of self-dual polynomials of degree $n$ in $\fifi[x]$ by 
\begin{equation} \label{rho} 
\rho:  \bmu \mapsto \prod_{f \in \Phi} \prod_{\gamma \in f} (x + \gamma)^{m_f}.
\end{equation}

Now consider the set of real-valued semisimple characters of $\U(n, \fifisq)$.  From Theorem \ref{RealChars} and (\ref{semisimple}), these are the characters of the form $\chi^{\blam}$ with $\blam \in \cP_n^{\Theta}$ where, for every $\varphi \in \Theta$, $\blam^{(\varphi)} = (1^{m_{\varphi}})$ for some $m_{\varphi} \geq 0$, and $\blam = \bar{\blam}$.  Let $\cS_{\Theta, n}$ denote this subset of $\cP_n^{\Theta}$ corresponding to real-valued semisimple characters of $\U(n, \fifisq)$.  Recall the bijection $\Delta: \cP_n^{\Theta} \rightarrow \cP_n^{\Phi}$ defined in Section \ref{Conj}.  It follows directly from the definition of $\Delta$ that if $\blam^{(\varphi)} = \blam^{(\bar{\varphi})}$ for every $\varphi \in \Theta$, then $\Delta(\blam)^{(f)} = \Delta(\blam)^{(\bar{f})}$ for every $f \in \Phi$.  Therefore, $\Delta$ gives a bijection from $\cS_{\Theta, n}$ to $\cS_{\Phi, n}$.  We summarize these observations as follows.
\begin{prop} \label{Semibijection}
For each real-valued semisimple character of $\U(n, \fifisq)$, map the corresponding $\blam \in \cS_{\Theta, n}$ to $\rho(\Delta(\blam))$, which is a self-dual polynomial in $\fifi[x]$ of degree $n$.  This map is a bijection, that is, the map
$$ \chi^{\blam} \mapsto \blam \mapsto \Delta(\blam) = \bmu \mapsto \rho(\bmu)$$  
gives a bijection from the set of real-valued semisimple characters of $\U(n, \fifisq)$ to the set of degree $n$ self-dual polynomials in $\fifi[x]$.
\end{prop}  

From the definition of the bijection $\rho \circ \Delta$ from $\cS_{\Theta, n}$ to the set of degree $n$ self-dual polynomials in $\fifi[x]$, for $\blam \in \cS_{\Theta, n}$, each partition $\blam^{(\varphi)} = (1^{m_{\varphi}})$ corresponds to a factor $\prod_{\gamma \in f} (x + \gamma)^{m_{\varphi}}$, where $\tilde{\partial}(f) = \varphi$ and $m_f = m_{\varphi}$, and $\tilde{\partial}$ is as defined in Section \ref{Conj}.  The choice of defining the map $\rho$ in (\ref{rho}) with factors in the form $(x+ \gamma)$ instead of $(x-\gamma)$ is so that we have the following fact, which is key in the counting argument in the proof of our main result.

\begin{prop} \label{sigmacount}
Let $\blam \in \cS_{\Theta, n}$, and let $\sigma \in L$ be the character as defined in Section \ref{Orbits}.  Through the bijection $\rho \circ \Delta$ from $\cS_{\Theta, n}$ to the set of degree $n$ self-dual polynomials in $\fifi[x]$, the partition $\blam^{( \{ \sigma \} )} = (1^{m_{\sigma}})$ corresponds to the factor $(x-1)^{m_{\sigma}}$.  In particular, the self-dual polynomial $(\rho \circ \Delta)(\blam)$ has constant term $(-1)^{m_{\sigma}}$.
\end{prop}
\begin{proof}  It is enough to show that $\tilde{\partial}([-1]) = [\sigma]$, or just that $\partial(-1) = \sigma$, which is the same since $[-1]$ and $[\sigma]$ are singleton $F$-orbits.  Let $\alpha$ be a multiplicative generator for the group $T_{n!}$, and let $\zeta$ be a primitive $|T_{n!}|$-th root of unity.  Then $\alpha^{|T_{n!}|/2} = -1 \in T_{n!}$, and $\zeta^{|T_{n!}|/2} = -1 \in \real$, and so by definition of $\partial$, $\partial(-1)$ is the character in $\widehat{T}_{n!}$ defined by $\alpha \mapsto -1$.  This is exactly the evaluation of $\sigma \in L$ as an element of $\widehat{T}_{n!}$, and thus $\partial(-1) = \sigma$.
\end{proof}

\section{Proof of the main theorem} \label{Main}

In order to calculate the Frobenius-Schur indicator of a real-valued semisimple character of $\U(n, \fifisq)$, by Theorem \ref{FScentral}, we must calculate the value of the central character at $\beta I$, where $\beta$ is a generator for $T_1$.  We may calculate the value of the central character by applying the formula obtained in Theorem \ref{CentralValues}.  In the case of a real-valued character, the formula in Theorem \ref{CentralValues} reduces to a particularly simple form, which we obtain in the next three Lemmas.

\begin{lemma} \label{FirstProdRed} Let $\blam \in \cP_n^{\Theta}$ such that $\bar{\blam} = \blam$.  Define $\omega^*_{\blam} \in L_1$ by
$$\omega^*_{\blam} =  \prod_{ \varphi \in \Theta \atop{\varphi \neq \bar{\varphi} \text{ or } \atop{|\varphi| \text{ even }}}} \prod_{\xi \in \varphi} \xi^{|\blam^{(\varphi)}|}.$$
Then, for any $a \in T_1$, we have $\omega^*_{\blam}(a)_1 = 1$.
\end{lemma}
\begin{proof}  First note that since for any $\varphi \in \Theta$, the character $\prod_{\xi \in \varphi} \xi$ is in $L_1$, it follows that $\omega^*_{\blam} \in L_1$.
 
First, if $\varphi \neq \bar{\varphi}$, then both $\varphi$ and $\bar{\varphi}$ appear in the product defining $\omega^*_{\blam}$, and $\blam^{(\varphi)} = \blam^{(\bar{\varphi})}$, since $\bar{\blam} = \blam$.  Now, for any $\xi \in \varphi$, we have $\xi^{-1} = \bar{\xi} \in \bar{\varphi}$.  It follows that for this case, we have, for any $a \in T_1$,
$$ \prod_{\xi \in \varphi} \xi^{|\blam^{(\varphi)}|} (a)_1 \prod_{\xi \in \bar{\varphi}} \xi^{|\blam^{(\bar{\varphi})}|}(a)_1 = 1.$$  

Now suppose that $\varphi = \bar{\varphi}$ and $|\varphi|$ is even.  Since the only elements of $L$ satisfying $\bar{\xi} = \xi$ are ${\bf 1}$ and $\sigma$, which are both singleton orbits in $\Theta$, we know that the elements of $\varphi$ may be paired, each $\xi \in \varphi$ with $\bar{\xi} = \xi^{-1} \in \varphi$.  So, in this case we have
$$ \prod_{\xi \in \varphi}  \xi^{|\blam^{(\varphi)}|} = {\bf 1}.$$
We have accounted for all terms in the product defining $\omega^*_{\blam}$, and so we have $\omega^*_{\blam}(a)_1 = 1$ for every $a \in T_1$.
\end{proof}

\begin{lemma} \label{OddPart} Let $\varphi \in \Theta$ be such that $\bar{\varphi} = \varphi$ and $|\varphi|$ is odd.  Then we must have either $\varphi = \{{\bf 1}\} $ or $\varphi = \{\sigma\}$.
\end{lemma}
\begin{proof} Suppose that $|\varphi| = 2k+1$ for some $k \geq 0$, so if $\xi \in \varphi$, then $2k+1$ is the smallest integer $i$ such that $\xi \circ F^i = \xi$.  Since we are assuming that $\varphi = \bar{\varphi}$, then there is an integer $j$, $0 \leq j < 2k+1$ such that 
\begin{equation} \label{xi}
\bar{\xi} = \xi \circ F^j. 
\end{equation} 
By composing both sides of (\ref{xi}) with $F^j$ we have $\bar{\xi} \circ F^j = \xi \circ F^{2j}$, and by conjugating both sides of (\ref{xi}), we have $\bar{\xi} \circ F^j = \xi$.  So, $\xi = \xi \circ F^{2j}$.  Letting $i = 2k+1-j$, we have $\bar{\xi} \circ F^i = \xi$, and by a similar argument as above we obtain $\bar{\xi} = \bar{\xi} \circ F^{2i}$.  However, either $i \leq k$ or $j \leq k$, and since $F$ has order $2k+1$ when acting on $\xi$ and on $\bar{\xi}$, then we must have $j=0$ (and $i = 2k+1$).  From (\ref{xi}), we must have $\bar{\xi} = \xi$, and in particular $\xi$ is $F$-stable.  The only such choices for $\xi$ are $\xi = {\bf 1}$ or $\xi = \sigma$.
\end{proof}

\begin{lemma} \label{ProdRed} Let $\blam \in \cP_n^{\Theta}$ such that $\bar{\blam} = \blam$, and let $\omega_{\blam}$ denote the character on $T_1$ defined in Theorem \ref{CentralValues}.  Let $\beta$ be a multiplicative generator for $T_1$.  Then,
$$ \omega_{\blam}(\beta)_1 = (-1)^{|\blam^{(\{\sigma\})}|}.$$
\end{lemma}
\begin{proof}  From the definition of $\omega_{\blam}$, we have 
$$ \omega_{\blam}(\beta)_1 = \prod_{\varphi \in \Theta} \prod_{\xi \in \varphi} \xi^{|\blam^{(\varphi)}|}(\beta)_1.$$ 
Applying Lemmas \ref{FirstProdRed} and \ref{OddPart}, this product reduces to
$$ \omega_{\blam}(\beta)_1 = \sigma(\beta)_1^{|\blam^{(\{\sigma\})}|}.$$
From the definition of the character $\sigma$, $\sigma(\beta)_1 = -1$, giving the result.
\end{proof}

We now prove the main result.  The main idea is that for a real-valued semisimple character $\chi$ of $\U(n, \fifisq)$ (for $q$ odd), through Proposition \ref{sigmacount} and Lemma \ref{ProdRed}, the sign of $\chi(\beta I)$ is the same as the sign of the constant term of the corresponding self-dual polynomial.

\begin{thm} \label{MainThm} Let $n = 2m$, and let $q$ be odd.  The number of semisimple symplectic characters of ${\rm U}(n, \fifisq)$ is $q^{m-1}$.
\end{thm}
\begin{proof} From Theorem \ref{FScentral}(2), Theorem \ref{CentralValues}, and Lemma \ref{ProdRed}, we must show that the number of real-valued semisimple characters $\chi^{\blam}$ of $\U(2m, \fifisq)$ such that $|\blam^{(\{\sigma\})}|$ is odd is $q^{m-1}$.

In Proposition \ref{Semibijection}, we give a one-to-one correspondence between real-valued semisimple characters of $\U(n,\fifisq)$ and self-dual polynomials over $\fifi$ of degree $n$.  From Proposition \ref{sigmacount}, in this correspondence, the partition $\blam^{(\{\sigma\})} = (1^{m_{\sigma}})$ corresponds to the factor $(x-1)^{m_{\sigma}}$, and the constant term of the self-dual polynomial corresponding to $\blam$ is $(-1)^{m_{\sigma}}$.  Thus, the number of real-valued semisimple characters of $\U(n, \fifisq)$ such that $m_{\sigma} = |\blam^{(\{\sigma\})}|$ is odd is equal to the number of self-dual polynomials in $\fifi[x]$ of degree $n$ with constant term equal to $-1$.

Suppose $g \in \fifi[x]$, $g(x) =  x^n + a_{n-1} x^{n-1} + \cdots a_1 x + a_0$, is a self-dual polynomial of degree $n = 2m$, and we assume also that $a_0 = -1$.  By (\ref{DualPoly}), we also have
\begin{equation} \label{count}
g(x) = -x^n  g(x^{-1}) = x^n - a_1 x^{n-1} - \cdots - a_{n-1} x - 1.
\end{equation}
Since $n = 2m$ is even, (\ref{count}) implies that we must have $a_m= -a_m$, and so $a_m = 0$.  For $1 \leq i \leq m-1$, we may let $a_i$ be any of $q$ elements of $\fifi$, and then each $a_{n-i} = -a_i$ is determined.  This gives a total of $q^{m-1}$ polynomials.
\end{proof}

As proven in \cite[Theorem 4.4]{GoVi08}, there are exactly $q^m + q^{m-1}$ real-valued semisimple characters of $\U(2m, \fifi)$ when $q$ is odd.  It follows from Theorem \ref{MainThm} that there are exactly $q^m$ semisimple orthogonal characters of $\U(2m, \fifi)$ when $q$ is odd.  As mentioned at the end of Section \ref{RegSemi}, it follows from \cite[Corollary 3.1]{Vi09} that there are exactly $q^{m-1}$ symplectic regular characters and $q^m$ orthogonal regular characters of $\U(2m, \fifi)$ when $q$ is odd.

\end{document}